\documentclass[10pt]{amsart}
\usepackage{bm}
\usepackage{amsmath}
\usepackage{amsthm}
\usepackage{color}
\usepackage{amscd}
\usepackage{amsfonts}
\usepackage{graphicx}
\usepackage{epsfig}
\usepackage{amssymb,latexsym}
\def\bbu{\bm{u}}
\def\bbx{\bm{x}}
\def\bu{\mathbf{u}}

\def\bx{\mathbf{x}}
\def\bz{\mathbf{z}}

\def\bn{\mathbf{n}}
\def\by{\mathbf{y}}
\def\be{\mathbf{e}}

\def\bbB{\mathbb B}

\def\bbf{\mathbf{f}}

\def\bbS{\mathbb{S}}
\def\R{\mathbb{R}}

\def\bQ{\mathbf{Q}}
\def\F{\mathbf F}
\def\A{\mathcal A}
\def\C{\mathcal C}
\def\id{\mathbf{1}}
\def\bcero{\mathbf 0}

\newtheorem{theorem}{Theorem}[section]
\newtheorem{corollary}[theorem]{Corollary}
\newtheorem{definition}{Definition}[section]
\newtheorem{lemma}[theorem]{Lemma} 
\newtheorem{proposition}[theorem]{Proposition}

\numberwithin{equation}{section}

\begin{document}

\title[Discontinuous ODE systems]{Beyond Peano's theorem: a variational look at discontinuous ODE systems}
\author{Pablo Pedregal}
\address{Departamento de Matemáticas. Universidad de Castilla La
Mancha. Campus de Ciudad Real, Spain} \email{pablo.pedregal@uclm.es}
\thanks{Supported by grants PID2023-151823NB-I00, and  SBPLY/23/180225/000023}\subjclass[2020]{}

\date{} % delete this line to display the current date
\begin{abstract}
We propose a framework to define solutions of ODE systems under a novel condition that goes well beyond the usual continuity condition required in the classical theory of ODEs (Peano's or Picard's theorems). We illustrate our results with some simple but enlightening examples, including some facts about Sobolev fields, and mention some relevant questions to proceed with this analysis further. 
\end{abstract}
%\end{abstract}
%%% BEGIN DOCUMENT
\maketitle
\section{Introduction}
The existence theory for initial-value, Cauchy problems for (non-linear, autonomous) ODE systems of the form
\begin{equation}\label{primera}
\bx'(t)=\bbf(\bx(t))\hbox{ in }(0, T),\quad \bx(0)=\bx_0\in\R^N,\quad \bbf(\bx):\Omega\subset\R^N\to\R^N,
\end{equation}
is one of those classic fields in Analysis that has hardly changed with the pass of time. Every analyst has his or her own favorite treatise in the field. There are hundreds of them, some utilized by whole generations of analysts. 

Two of the basic such existence theorems require as an essential ingredient the (Lipschitz) continuity of the right-hand side $\bbf$ in \eqref{primera}. Even if there has been attempts to generalize such continuity conditions to adapt the problem to more general scenarios and needs (check the Introduction in \cite{ambrosio} for some comments and references in this regard), a direct approach to problem \eqref{primera} as such in the absence of continuity for $\bbf$ has not been developed, to the best of our knowledge. Some variants essentially rely on the classical Lipschitz condition, or exploit in a fundamental way the particular structure of the system.  
There are, however, at least two main frameworks where non-smooth/non-continuous right-hand sides in \eqref{primera} have been considered systematically: 
\begin{enumerate}
\item The needs of PDEs of Mathematical Physics have pushed to more general consistent theories for Cauchy problems like the one in \eqref{primera}. This is particularly relevant for transport equations or fluid dynamics. 
Since the seminal work \cite{dipernalions}, some further development has taken place. \cite{ambcrip}, \cite{ambrosio} are nice references where interested readers may learn about these new ideas. 
\item Discontinuous, piece-wise continuous right-hand sides are being studied from the perspective of global dynamics. References like \cite{filippov}, \cite{smirnov} are good places where more general differential systems are examined. Look at \cite{llibreponce} for an example in which important features about global dynamics of some discontinuous differential systems are studied. The amount of literature in this area is growing very rapidly. 
\end{enumerate}
Our proposal here breaks up definitely with the continuity barrier of the right-hand side $\bbf$ in \eqref{primera}, and allows for a generalized concept of solution. 

\begin{definition}\label{basica}
Let $\Omega\subset\R^N$ be an open set, and $\bbf(\bx):\Omega\to\R^N$ be a Borel field in $\R^N$, $N>1$.  
\begin{enumerate}
\item We say that $\bbf$ is self-continuous at $\bx\in\Omega$ if 
%it admits a representative (not relabeled) such that
\begin{equation}\label{limite}
\lim_{\epsilon\to0^+}\bbf\left(\bx+\epsilon\bbf(\bx)\right)=\bbf(\bx). 
\end{equation}
It is self-continuous in $\Omega$ if \eqref{limite} holds for every $\bx\in\Omega$. 
\item In more general and flexible terms, we say that such a map $\bbf$ is self-continuous at $\bx\in\Omega$, if there is a $\C^1$-curve
\begin{equation}\label{curva}
\phi(\epsilon; \bx):[0, \epsilon_0)\to\Omega,\quad \phi(0; \bx)=\bx, \phi'(0; \bx)=\bbf(\bx), \epsilon_0=\epsilon_0(\bx)>0,
\end{equation}
such that
\begin{equation}\label{germen}
\lim_{\epsilon\to0^+}[\phi'(\epsilon; \bx)-\bbf(\phi(\epsilon; \bx))]=\bcero.
\end{equation}
\end{enumerate}
In a similar way, $\bbf$ is said to be self-continuous in $\Omega$ if it is so at every point $\bx\in\Omega$. 
\end{definition}
Condition \eqref{germen}, under  \eqref{curva}, implies that
$$
\lim_{\epsilon\to0}\bbf(\phi(\epsilon; \bx))=\bbf(\bx).
$$
It is interesting to realize how the curve $\phi(\epsilon; \bx)$ in \eqref{curva}, through condition \eqref{germen}, is like a ``germ" at $\bx$ for the solution of the initial differential system \eqref{primera}. The whole point is to assemble all those germs together to produce a solution curve globally in some finite interval. 

Note that \eqref{limite} is void (trivially true) in places where $\bbf$ vanishes. 

It should then be clear that we are in need of a more flexible, coherent concept of solution for a discontinuous ODE system \eqref{primera}. To introduce this new concept, we will be working in the classical Sobolev space $W^{1, 1}(0, T; \Omega)$ of absolutely continuous paths with image in $\Omega\subset\R^N$ in the interval $[0, T]$ (\cite{brezis}, \cite{leoni}).

\begin{definition}\label{generalizada}
Let $\bbf(\bx):\Omega\to\R^N$ be a Borel field in $\R^N$, $N>1$. 
An absolutely-continuous path 
$$
\bx(t):[0, T]\to\R^N,\quad \bx(0)=\bx_0, 
$$
is said to be a generalized solution of the (non-continuous) ODE system
$$
\bx'(t)=\bbf(\bx(t))\hbox{ in }(0, T),\quad \bx(0)=\bx_0,
$$
if there is a sequence of absolutely continuous paths 
$$
\{\bx_j\}\subset W^{1, 1}(0, T; \Omega),
$$ 
such that
\begin{equation}\label{errorj}
\bx'_j(t)=\bbf(\bx_j(t))+\varepsilon_j(t),
\end{equation}
and
$$
\bx_j\rightharpoonup\bx \hbox{ in }W^{1, 1}(0, T; \R^N), \quad \int_0^T|\varepsilon_j(t)|\,dt\to0, 
$$
as $j\to\infty$.
\end{definition}

This concept is, in principle, broader than taking a family of continuous approximations $\{\bbf_j\}$ of the discontinuous map $\bbf$, and considering the approximating paths $\{\bx_j\}$ determined by
$$
\bx'_j(t)=\bbf_j(\bx_j(t))\hbox{ in }(0, T),\quad \bx_j(0)=\bx_0.
$$
One would have to ensure, for the equivalence of both methods, at least the point-wise limit 
$$
\lim_{j\to\infty}\bbf_j(\bx)=\bbf(\bx)
$$
for every $\bx$, even for vectors in the discontinuity set of $\bbf$. The remarkable property of Definition \ref{generalizada} is that it only depends on the vector field $\bbf$ itself, and does not make any reference to the particular potential approximating sequence $\{\bbf_j\}$ and how this might be assembled depending on special properties of the underlying $\bbf$. 

On the other hand, if a certain absolutely continuous path $\bx$ is the weak (or strong) limit in $W^{1, 1}(0, T; \R^N)$ of a sequence of true solutions $\bx_j$ (in particular if $\bx_j\equiv\bx$ for all $j$)
$$
\bx'_j(t)=\bbf(\bx_j(t)),\quad \varepsilon_j\equiv0\hbox{ in \eqref{errorj}},
$$
then $\bx$ becomes a generalized solution, even if the difference 
$$
\bx'(t)-\bbf(\bx(t))
$$ 
is far from vanishing. We will emphasize this point later.

Our main existence result tights together both definitions above, and establishes that, under self-continuity, one can always define integral curves in a coherent way. The self-continuity property is, of course, a sufficient condition.
\begin{theorem}\label{central}
Let $\bbf(\bz):\R^N\to\R^N$ be self-continuous, and
$$
|\bbf(\bz)|\le C_1|\bz|+C_0,\quad C_1, C_0>0,
$$ 
for every $\bz\in\R^N$. Then the initial-value, Cauchy problem
$$
\bx'(t)=\bbf(\bx(t))\hbox{ in }(0, T),\quad \bx(0)=\bx_0
$$
admits at least one  generalized, absolutely continuous solution $\bx(t)$ for every $\bx_0\in\R^N$, and every positive $T$. 
\end{theorem}
Uniqueness of generalized solutions is out of the question. Indeed, non-uniqueness of generalized solutions is even more dramatic than with classical solutions for non-Lipschitzian maps. Dependence on initial conditions and/or parameters cannot, obviously, be continuous if $\bbf$ is not.

Our strategy is variational in the sense that we will look at the basic initial-value Cauchy problem
$$
\bx'(t)=\bbf(\bx(t))\hbox{ in }(0, T),\quad \bx(0)=\bx_0,
$$
through the variational problem consisting in minimizing the error functional
\begin{equation}\label{error}
E(\by)=\int_0^T|\by'(t)-\bbf(\by(t))|\,dt,\quad \by\in W^{1, 1}(0, T; \R^N), \by(0)=\bx_0.
\end{equation}
It is clear that solutions $\bx$ of the above Cauchy problem will correspond with minimizers $\by$ provided that $E(\by)=0$. This alternative viewpoint was introduced in a classic setting under continuity or even Lipschitz continuity in \cite{amatped}, and later further developed in \cite{amatped2}, and \cite{pedregal}. 

It is important to notice that error functional $E$ in \eqref{error} is not, in general, weak lower semicontinuous in $W^{1, 1}(0, T; \Omega)$ precisely because mapping $\bbf$ is not assumed to be continuous. This observation invalidates the use of the direct method of the Calculus of Variations (\cite{dacorogna}) to deal with the problem of minimizing $E$ under suitable end-point conditions. In fact, under the continuity of field $\bbf$, one can easily recover typical, classic existence results for problem \eqref{primera}. We will check this later to vindicate coherence of our proposal with classical existence theories. In addition, we caution readers about the easiness with which the continuity of $\bbf$ may slips through our arguments as we are so used to claim this property for the right-hand side in \eqref{primera}. 

Our ideas can straightforwardly be generalized (as in \cite{amatped}) to cover:
\begin{enumerate}
\item a local result where the time horizon $T$ of existence of solutions can only be ensured to be small if a linear upper bound on the size of $\bbf$ is not global; 
\item  a non-autonomous scenario with an explicit dependence of $\bbf$ on time $t$; and 
\item an implicit system of the form
$$
\F(t, \bx(t), \bx'(t))=\bcero\hbox{ in }(0, T),\quad \bx(0)=\bx_0.
$$
\end{enumerate}
Several simple examples are treated in Section \ref{ejemplos} to illustrate our results, and the self-continuity condition. In Section \ref{sob}, we briefly comment on the Sobolev case. 

\section{Proof of main result}
One elementary fact which we will be in need of along the proof of our main result below is the following.
\begin{lemma}\label{auxiliar1}
 Suppose that $\bbf(\bz):\R^N\to\R^N$ is measurable, and of linear growth
\begin{equation}\label{cotasup}
|\bbf(\bz)|\le C_1|\bz|+C_0,\quad C_1, C_0>0,
\end{equation}
for every $\bz\in\R^N$. If for $E$ in \eqref{error}, we have $E(\by)<+\infty$, then $\by\in L^\infty(0, T; \R^N)$ and
$$
|\by(t)|\le (TC_0+|\bx_0|+E(\by))e^{C_1T}
$$
for every $t\in(0, T)$.
\end{lemma}
\begin{proof}
We have
\begin{align}
|\by(t)-\bx_0|\le&\int_0^t|\by'(s)|\,ds\nonumber\\
\le &\int_0^T|\by'(s)-\bbf(\by(s))|\,ds+\int_0^t|\bbf(\by(s))|\,ds\nonumber\\
\le & E(\by)+C_1\int_0^t|\by(s)|\,ds+C_0.\nonumber
\end{align}
Conclude by Gromwall's lemma. 
\end{proof}
Next, we introduce some notation to facilitate our arguments. For $0<s<r<T$, we put
\begin{gather}
\A_{s, r, \bx_0}=\{\bx\in W^{1, 1}(s, r; \R^N): \bx(s)=\bx_0\}=\bx_0+\A_{s, r, \bcero},\nonumber\\
\A_{s, r, \bx_0, \bx_1}=\{\bx\in W^{1, 1}(s, r; \R^N): \bx(s)=\bx_0, \bx(r)=\bx_1\}=L_{s, r, \bx_0, \bx_1}+ W^{1, 1}_0(s, r; \R^N),\nonumber
\end{gather}
where $L_{s, r, \bx_0, \bx_1}$ is the linear path
\begin{gather}\label{lineal}
L_{s, r, \bx_0, \bx_1}(t)=\bx_0+\frac1{r-s}(\bx_1-\bx_0)(t-s),\quad t\in(r, s).
\end{gather}
Similarly, put
$$
E_{s, r}(\bx)=\int_s^r|\bx'(t)-\bbf(\bx(t))|\,dt
$$
for an absolutely continuous path $\bx$ in the interval $(s, r)$. 

For the auxiliary result that follows, we define the scalar function
$$
G(\bz): \Omega\to\R,\quad G(\bz)=\inf\{E_{0, r}(\bx): \bx\in\A_{0, r, \bx_0, \bz}\},
$$
for fixed positive $r>0$, and $\bx_0\in\Omega$. In its proof, as well as in the proof of our main result, we will be using the same fundamental idea which, on the other hand, was introduced and utilized in \cite{amatped} in a classical setting. 

\begin{lemma}\label{auxiliar2}
If the upper bound \eqref{cotasup} is valid for $\bbf$, then for every choice of $r>0$ and $\bx_0\in\Omega$, the above function $G(\bz)$ is continuous. 
\end{lemma}
As we shall see in the proof, it is even true that
$$
|G(\bz)-G(\by)|\le|\bz-\by|.
$$
\begin{proof}
We will divide the proof in two main steps.

For the first one, it is important to explicitly write the dependence of our function $G$ on $r$, namely
$$
G(r, \bz)=\inf\{E_{0, r}(\bx): \bx\in\A_{0, r, \bx_0, \bz}\},
$$
as we would like to show that for each fixed $\bz$, there is continuity on $r$. Vector $\bx_0$ is assumed to be fixed throughout. To that end, we put
\begin{gather}
G(r, \bz)=\inf\{E_{0, r}(\bx): \bx\in\A_{0, r, \bx_0, \bz}\},\nonumber\\
G(s, \bz)=\inf\{E_{0, s}(\by): \by\in\A_{0, s, \bx_0, \bz}\}.\nonumber
\end{gather}
Define the natural bijective transformation 
$$
\bx\in\A_{0, r, \bx_0, \bz}\mapsto\by\in\A_{0, s, \bx_0, \bz},\quad \by(\tau)=\bx(r\tau/s),
$$
through the natural change of time scale. 
After a few elementary calculations and a natural change of variables, we find that
$$
E_{0, s}(\by)=\int_0^r|\bx'(t)-\frac sr\bbf(\bx(t))|\,dt.
$$
From here, and bearing in mind \eqref{cotasup}, 
$$
E_{0, s}(\by)\le E_{0, r}(\bx)+\left|1-\frac sr\right| \left(C_1\int_0^r|\bx(t)|\,dt+C_0r\right).
$$
It is then true that
$$
G(s, \bz)\le E_{0, r}(\bx)+\left|1-\frac sr\right| \left(C_1\int_0^r|\bx(t)|\,dt+C_0r\right)
$$
for each individual $\bx\in\A_{0, r, \bx_0, \bz}$. If we now take the infimum in such paths $\bx$, by Lemma \ref{auxiliar1}, we conclude the existence of a constant $M$ (depending possibly on the time horizon $T$) with
$$
G(s, \bz)\le G(r, \bz)+\left|1-\frac sr\right| M.
$$
Since $r$ and $s$ are arbitrary, by interchanging their roles, we have that
$$
|G(s, \bz)-G(r, \bz)|\le M\max\left\{\left|1-\frac sr\right|, \left|1-\frac rs\right|\right\},
$$
which shows the claimed continuity. 

For the second step, consider $r$ given, and let $s<r$. Notice that for every $\by\in\R^N$ fixed, 
$$
G(r, \bz)\le\inf_\bx\{E_{0, s}(\bx): \bx\in\A_{0, s, \bx_0, \by}\}+\inf_\bx\{E_{s, r}(\bx): \bx\in\A_{s, r, \by, \bz}\}.
$$
In this inequality, we are using the fact that absolutely continuous paths can be defined independently in a finite union of disjoint open intervals as long as they remain globally continuous. Even further we can use the linear path $L\equiv L_{s, r, \by, \bz}$ in \eqref{lineal} in the second term above, to deduce 
$$
G(r, \bz)\le G(s, \by)+\int_s^r\left|\frac1{r-s}(\bz-\by)-\bbf(L(t))\right|\,dt. 
$$
By introducing a similar constant $M$ as in the preceding step, we see that
$$
G(r, \bz)\le G(s, \by)+|\bz-\by|+M(r-s). 
$$
If we let $s\to r$, by the continuity shown in the previous step, 
$$
G(r, \bz)\le G(r, \by)+|\bz-\by|;
$$
and by interchanging the roles of $\bz$ and $\by$, we finish the proof.
\end{proof}
We are now ready to prove our main theorem, which we restate for the convenience of readers. 
\begin{theorem}
Let $\bbf(\bz):\R^N\to\R^N$ be self-continuous, and
$$
|\bbf(\bz)|\le C_1|\bz|+C_0,\quad C_1, C_0>0,
$$ 
for every $\bz\in\R^N$. Then the initial-value, Cauchy problem
$$
\bx'(t)=\bbf(\bx(t))\hbox{ in }(0, T),\quad \bx(0)=\bx_0
$$
admits at least one  generalized, absolutely continuous solution $\bx(t)$ for every $\bx_0\in\R^N$, and every positive $T$. 
\end{theorem}
\begin{proof}
We divide the proof in two parts. For the first one, 
consider, as introduced above, the family of functionals and feasible classes of paths
$$
E_{s, r}(\bx)=\int_s^r|\bx'(t)-\bbf(\bx(t))|\,dt,\quad \bx\in \A_{s, r, \bx_0},
$$
and put
$$
m(r):[0, T]\to\R^+,\quad m(r)=\inf \{E_{0, r}(\bx): \bx\in \A_{0, r, \bx_0}\},
$$
for $0<r<T$. Note that each functional $E_{s, t}$ is well-defined because its integrand, as the composition of a measurable function and a continuous path, is measurable. 

It is evident that $m(0)=0$. For positive $r>0$, and $h>0$, let us try to relate $m(r+h)$ to $m(r)$. To this end, we realize, as we have already argued before, that
\begin{equation}\label{igualdad}
m(r+h)=\inf_{\bz\in\R^N}\left\{\ \inf\{E_{0, r}(\bx): \bx\in\A_{0, r, \bx_0, \bz}\}+\inf\{E_{r, r+h}(\bx): \bx\in\A_{r, r+h, \bz}\}\ \right\}.
\end{equation}
On the other hand, note that
\begin{equation}\label{nestedinf}
m(r)=\inf\{E_{0, r}(\bx): \bx\in\A_{0, r, \bx_0}\}=\inf_{\bz\in\R^N}\left\{\inf\{E_{0, r}(\bx): \bx\in\A_{0, r, \bx_0, \bz}\}\right\}.
\end{equation}
Given that the final end-point conditions for minimizing paths for the inner infimum in \eqref{nestedinf} are uniformly bounded (Lemma \ref{auxiliar1}) and the inner expression as a function of $\bz$ is continuous (Lemma \ref{auxiliar2}), the outer infimum in $\bz\in\R^N$ in that same equality must be, in fact, a minimum. Let $\bz_r$ be such vector minimizer (depending possible on $r$) in such a way that
$$
m(r)=\inf\{E_{0, r}(\bx): \bx\in\A_{0, r, \bx_0, \bz_r}\}.
$$
Going back to \eqref{igualdad}, by setting $\bz=\bz_r$, we can conclude that
\begin{equation}\label{suma}
m(r+h)\le m(r)+\inf\{E_{r, r+h}(\bx): \bx\in\A_{r, r+h, \bz_r}\}.
\end{equation}
In particular, the linear path 
\begin{equation}\label{linear}
\bx(t)=\bz_r+(t-r)\bbf(\bz_r),\quad t\in(r, r+h), 
\end{equation}
is feasible in the infimum in the right-hand side of \eqref{suma}, and thus
$$
m(r+h)\le m(r)+\int_r^{r+h}|\bbf(\bz_r)-\bbf(\bz_r+(t-r)\bbf(\bz_r))|\,dt.
$$
Bearing in mind that the integrand for the error functional is non-negative, and recalling that $h>0$, we finally find that
$$
0\le\frac1h(m(r+h)-m(r))\le \frac1h\int_r^{r+h}|\bbf(\bz_r)-\bbf(\bz_r+(t-r)\bbf(\bz_r))|\,dt.
$$
The self-continuity of $\bbf$ at $\bz_r$ implies clearly that the value function $m(r)$ admits a right derivative at every point, and that it vanishes. Since $m(0)=0$, we conclude that $m$ identically vanishes in $[0, T]$. In particular $m(T)=0$, i.e. 
\begin{equation}\label{conclusion}
\inf\left\{\int_0^T |\bx'(t)-\bbf(\bx(t))|\,dt: \bx\in\A_{0, T, \bx_0}\right\}=0.
\end{equation}

For the second part of the proof, we start by interpreting \eqref{conclusion} in the sense that there is a sequence of absolutely continuous paths $\{\bx_j\}\subset\A_{0, T, \bx_0}$ such that
$$
E_{0, T}(\bx_j)\searrow0\hbox{ as }j\to\infty.
$$
For the sake of simplicity, we just write $E$ to mean $E_{0, T}$, and so $E(\bx_j)\to0$. We claim 
that there is $\bx\in\A_{0, t, \bx_0}$ such that $\bx_j\rightharpoonup\bx$ in $W^{1, 1}(0, T; \R^N)$.
To show this, let $J\subset[0, T]$ be arbitrary. Just as we have argued earlier, 
$$
\int_J|\bx'_j(t)|\,dt\le E(\bx_j)+C_1\int_J|\bx_j(s)|\,ds+C_0|J|.
$$
Further, by Lemma \ref{auxiliar1}, 
$$
\int_J|\bx'_j(t)|\,dt\le E(\bx_j)+(C_1e^{C_1T}(TC_0+|\bx_0|+E(\bx_j))+C_0)|J|.
$$
Given that $E(\bx_j)\to0$, the arbitrariness of the subset $J$ in this inequality, implies not only that $\{\bx_j\}$ is uniformly bounded in $W^{1, 1}(0, T; \R^N)$, but also that $\{\bx'_j\}$ is equi-integrable. Consequently, we conclude our claim $\bx_j\rightharpoonup\bx$ in $W^{1, 1}(0, T; \R^N)$. 

A suitable interpretation of our two steps leads directly to Definition \ref{generalizada}, if we set
$$
\varepsilon_j(t)=\bx'_j(t)-\bbf(\bx_j(t)), 
$$
and recall that 
$$
\int_0^T\varepsilon_j(t)\,dt=E(\bx_j)\to0.
$$

If, more in general, there is a path $\phi(t; \bx)$ as in Definition \ref{basica} possibly not always identical to the line 
$$
\phi(t; \bx)=\bx+t\bbf(\bx),
$$ 
such path is feasible in 
 the infimum in the right-hand side of \eqref{suma}, and replaces the linear path in \eqref{linear}.  Thus
$$
m(r+h)\le m(r)+\int_r^{r+h}|\phi'(t; \bz_r)-\bbf(\phi(t; \bz_r))|\,dt,
$$
and we proceed as above, under \eqref{germen}, to conclude that $m(T)=0$. The second step would be identical. 
\end{proof}

\section{A legitimate extension}
According to \cite{filippov}, and as it could not be accepted otherwise, one of the most important requirements for a fruitful extension of the concept of solution for a discontinuous differential system is that in the case the right-hand side is continuous, the only possibility of any kind of generalized solution is the classical one. 

\begin{corollary}
Suppose $\bbf(\bx):\Omega\subset\R^N$ is continuous. Then every generalized solution of the corresponding system $\bx'=\bbf(\bx)$ is a classical solution. 
\end{corollary}
\begin{proof}
Suppose that a certain path 
$$
\bx(t):[0, T]\to\Omega
$$ 
is a generalized solution according to Definition \ref{generalizada}, that is there is a sequence of absolutely continuous paths $\{\bx_j\}$ such that
$$
\bx_j\rightharpoonup\bx\hbox{ in }W^{1, 1}(0, T; \Omega),\quad E(\bx_j)\to0.
$$
This time the integral functional 
$$
E(\bz)=\int_0^T|\bz'(t)-\bbf(\bz(t))|\,dt
$$
is weak lower-semicontinuous in $W^{1, 1}(0, T; \R^N)$ because its integrand
$$
\phi(\bz, \bu)=|\bu-\bbf(\bz)|
$$
is convex in the variable $\bu$, and continuous in $\bz$ (\cite{dacorogna}). Hence
$$
0\le E(\bx)\le\liminf_{j\to\infty}E(\bx_j)=0,
$$
and the differential system holds for $\bx$ at a.e. $t\in[0, T]$. Once again, since $\bbf$ is continuous, it must hold for every time $t$, and $\bx$ is a classical solution. 
\end{proof}

In the preceding proof, the weak lower semicontinuity of functional $E$ is unavoidable. In general, under the self-continuity condition, all we can ensure is that \eqref{conclusion} holds. Consequently, there is a minimizing sequence $\{\bx_j\}$ of paths in $\A_{0, T, \bx_0}$ with $E(\bx_j)\searrow0$ as $j\to\infty$; and there is a weak limit 
$$
\bx_j\rightharpoonup\bx\hbox{ in }W^{1, 1}(0, T; \R^N).
$$
If the mapping $\bbf$ is not continuous, functional $E$ cannot be ensured to be weak lower semicontinuous, and thus there are two possibilities for the limit path $\bx\in\A_{0, T, \bx_0}$:
\begin{enumerate}
\item either $E(\bx)=0$; or 
\item $E(\bx)>0$.
\end{enumerate}
In the first case, we conclude that
$$
\bx'(t)=\bbf(\bx(t))\hbox{ for a.e. }t\in(0, T),
$$
and the ODE system \eqref{primera} holds point-wise for a.e. $t\in(0, T)$. Indeed, whenever this is the case for a certain absolutely continuous path $\overline\bx$, i. e.
$$
\overline\bx'(t)=\bbf(\overline\bx(t))\hbox{ for a.e. }t\in (0, T),
$$
then $\overline \bx$ becomes a generalized solution according to Definition \ref{basica}. 

For the second case, the limit path $\bx$ will be far from being a point-wise solution, though the ODE system \eqref{primera} might or might not admit other such solutions. We will see next some elementary examples to clarify all these possibilities. 

\section{Some illustrative examples}\label{ejemplos}
The first mandatory observation is that a.e. solutions of \eqref{primera}
$$
\bx'(t)=\bbf(\bx(t))\hbox{ a.e. in }(0, T),\quad \bx(0)=\bx_0\in\R^N,\quad \bbf(\bx):\Omega\subset\R^N\to\R^N,
$$
are generalized solutions, regardless of whether $\bbf$ is self-continuous, or even if it is not defined everywhere. It may admit additional solutions depending on how $\bbf$ is defined and whether it is self-continuous. This is the case for the simple example
$$
\bbf(\bx)=(-1, x_1/|x_1|)=\begin{cases}(-1, 1),& x_1>0,\cr (-1, -1),& x_1<0,\end{cases},\quad \bx=(x_1, x_2), 
$$
which admits generalized solutions starting at $x_1=0$, regardless of how $\bbf$ is defined on $x_1=0$. If it is defined in this vertical line in a way that $\bbf$ becomes self-continuous (for example by putting $\bbf(0, x_2)=(0, 1)$ for all $x_2$), then there are more solutions starting at this vertical axis. The situation is drastically distinct with the example
$$
\bbf(\bx)=(-x_1/|x_1|, 1)=\begin{cases}(-1, 1),& x_1>0,\cr (1, 1),& x_1<0,\end{cases}
$$
for which  existence of generalized solutions through the vertical axis $x_1=0$ depends on how $\bbf$ is defined there. The only possibility is to set 
$$
\bbf(0, x_2)=(0, f(x_2))
$$ 
for a continuous scalar function $f$. See below.

The following examples may further help in beginning to have some intuition on the self-continuity condition. They are academic and elementary, but instructive. Integral curves are easily deduced in all cases. 
\begin{enumerate}
\item\label{uno} For $\bx\in\R^N$, consider the discontinuous vector field
$$
\bbf(\bx)=\frac1{|\bx|}\bx,\bx\neq\bcero,\quad \bbf(\bcero)=\bn,
$$
where the unitary vector $\bn$ is chosen arbitrarily in $\R^N$. Every such $\bbf$ is easily checked to be self-continuous, and hence there is, obviously, an integral curve starting at the origin as well. If $\bn$ is not unitary (except for the null vector), the resulting field is not self-continuous, yet it admits the same generalized solution starting at zero. If $\bn=\bcero$, $\bbf$ becomes trivially self-continuous as pointed out earlier. 
\item For $\bx\in\R^2$, the discontinuous field
$$
\bbf(\bx)=\frac1{|\bx|}\bQ\bx,\bx\neq\bcero,\quad \bbf(\bcero)\neq\bcero,\quad \bQ=\begin{pmatrix}0&1\cr-1&0\end{pmatrix},
$$
is not self-continuous. It is clear that there is no consistent way to define any kind of solution starting at the origin. The only possibility is to transform the origin in an equilibrium point by setting $\bbf(\bcero)=\bcero$. 

For the three-dimensional variant
$$
\bbf(\bx)=\frac1{|\hat\bx|}\bQ\bx,\hat\bx\neq\bcero,\quad \bx=(x_1, x_2, x_3), \hat\bx=(x_1, x_2, 0), \bQ=\begin{pmatrix}
0&1&0\cr-1&0&0\cr0&0&1\end{pmatrix},
$$
and
$$
\bbf(0, 0, x_3)=(0, 0, f(x_3)),
$$
with $f$ continuous, it is readily seen that this $\bbf$ is self-continuous. The only generalized integral curves starting at the vertical axis cannot leave such axis. 
\item For the field
$$
\bbf(\bx)=\frac1{|\bx|(|\bx|-1)}\bQ\bx,\quad |\bx|\neq1, \bx\in\R^2, 
$$
to become self-continuous, we need to define it on the unit circle $|\bx|=1$ in such a way that it becomes an invariant manifold for it; and, as above, we also need to put $\bbf(\bcero)=\bcero$. A suitable extension to $\R^3$ would require an invariant tube (cylinder). 
\item The vector field
$$
\bbf(\bx)=(\sin(1/|\bx|), \cos(1/|\bx|)), \bx\neq\bcero,\quad \bbf(\bcero)\neq\bcero,
$$
is not self-continuous. This is similar to the previous example. The only possibility to convert it to a self-continuous field is to set $\bbf(\bcero)=\bcero$. 
\item The vector field
$$
\bbf(\bx)=\begin{cases}(-x_1/|x_1|, 1),& x_1\neq0,\cr (1, 0),& x_1=0,\end{cases},\quad \bx=(x_1, x_2), 
$$
is not self-continuous at the vertical axis. There is no way to define integral curves starting at such points. However, the variant
$$
\bbf(\bx)=\begin{cases}(-x_1/|x_1|, 1),& x_1\neq0,\cr (0, f(x_2)),& x_1=0,\end{cases},\quad \bx=(x_1, x_2), 
$$
is self-continuous for every continuous scalar function $f$, and its integral curves starting at the vertical axis stay there.
In addition, integral curves starting off the vertical axis, eventually enter into it, and from that moment on, remain there. They are defined for all time. 
\item The vector field
$$
\bbf(\bx)=\begin{cases}(x_1/|x_1|, 1),& x_1\neq0,\cr (0, f(x_2)),& x_1=0,\end{cases},\quad \bx=(x_1, x_2), 
$$
is not continuous, but it is self-continuous, for every continuous scalar function $f$. We evidently see non-uniqueness of solutions starting at the vertical axis. 
The variant
$$
\bbf(\bx)=\begin{cases}(1, 1),& x_1\ge0,\cr (-1, 1),& x_1<0,\end{cases},\quad \bx=(x_1, x_2), 
$$
is also self-continuous as well as
$$
\bbf(\bx)=\begin{cases}(1, 1),& x_1>0,\cr (-1, 1),& x_1\le0.\end{cases}
$$
In these two cases, the vertical solutions are lost. 
\item The vector field
$$
\bbf(\bx)=\begin{cases}(0, x_1/|x_1|),& x_1\neq0,\cr (1, 0),& x_1=0,\end{cases},\quad \bx=(x_1, x_2), 
$$
is not self-continuous at the vertical axis. Yet, there are generalized integral curves of the form 
$$
\bx_+(t)=(t, 0),\quad \bx_-(t)=(-t, 0),
$$
for $t\ge0$, because they are limits of true integral curves 
$$
\bx_{+, j}(t)=(t, 1/j),\quad \bx_{-, j}(t)=(-t, -1/j).
$$
The variation
$$
\bbf(\bx)=\begin{cases}(0, x_1/|x_1|),& x_1\neq0,\cr (0, f(x_2)),& x_1=0,\end{cases},\quad \bx=(x_1, x_2), 
$$
is self-continuous, and, in addition to the previous integral curves, there is also a vertical integral curve running through the vertical axis. 
\end{enumerate}

A standard, general situation for a unitary, discontinuous vector field corresponds to
$$
\F(\bx)=\frac1{|\bbf(\bx)|}\bbf(\bx),
$$
where $\bbf(\bx)$ is a continuous vector field in $\R^N$. The points of discontinuity of $\F$ are exactly the equilibria of $\bbf$. It would be interesting to explore the self-continuity of $\F$ depending on the nature of the underlying equilibrium point for $\bbf$ even in the case that $\bbf$ is linear. This issue goes, however, beyond the scope of this initial contribution. 

\section{The case of Sobolev fields}\label{sob}
After having gained some familiarity with the self-continuity condition, we would like to address the case of a Sobolev field 
$$
\bbu\in W^{1, p}(\Omega; \R^N),\quad \bbu(\bbx):\Omega\subset\R^N\to\R^N, p\ge1.
$$
Note the change of notation $\bbf\mapsto\bbu$, $\bx\mapsto\bbx$ with respect to previous sections. Of course, if $p>N$ field $\bbu$ is continuous, and hence self-continuous. On the other hand, it is not difficult to see that for $p<N$, one cannot expect fields to be self-continuous in general, as the following example shows. 

Consider
$$
\bbu(\bbx)=|\bbx|^\alpha\bbx,\quad \bbx\in\Omega\equiv\bbB,\quad \alpha<-1.
$$
We would like this field to belong to $W^{1, p}(\bbB; \R^N)$ for $p<N$,
where $\bbB$ is the unit ball in $\R^N$, and exponent $\alpha$ is chosen in the interval $(-N/p, -1)$. Then
$$
\nabla\bbu(\bbx)=|\bbx|^\alpha\left(\alpha\frac\bbx{|\bbx|}\otimes\frac\bbx{|\bbx|}+\id\right),
$$
where $\id$ represents the identity matrix of dimension $N\times N$. 
Because the integral
$$
\int_{\bbB}|\bbx|^{p\alpha}\,d\bbx
$$
is finite for the given range of exponents, we conclude that indeed $\bbu\in W^{1, p}(\Omega; \R^N)$. However, it is also clear, because $\alpha<-1$, that none of the limits 
$$
\lim_{r\to0^+}\bbu(r\be),\quad |\be|=1, 
$$
exists. Hence $\bbu$ cannot be self-continuous at the origin. 

It is well-known (\cite{leoni}) that Sobolev fields are absolutely continuous along lines through almost every point. However, the proof of our main existence result (Theorem \ref{central}) does not seem to be valid if the self-continuity condition \eqref{limite} is only assumed for a.e. vector $\bx$ instead of for every such point. The following result yields an easy, sufficient condition under which Sobolev fields can be expected to be self-continuous. 

\begin{proposition}
Let $\bbu$ be a Sobolev field in $\Omega\subset\R^N$. If for every point $\bbx_0\in\Omega$ of discontinuity of $\bbu$, there is $\rho>0$ with the ball $\bbB_\rho(\bbx_0)\subset\Omega$ centered at $\bbx_0$ and radius $\rho$, such that the function
$$
\bbx\mapsto|\bbx-\bbx_0|^{-N}|\nabla\bbu(\bbx)(\bbx-\bbx_0)|
$$
belongs to $L^1(\bbB_\rho(\bbx_0))$, then $\bbu$ admits a representative which is self-continuous.
\end{proposition}
\begin{proof}
For a Sobolev field $\bbu$, define the sections
$$
u_{\bbx_0, \be}(r)=\bbu(\bbx_0+r\be),\quad r\in(0, \rho),
$$
for unit, arbitrary vector $\be$. It is immediate to calculate 
$$
u'_{\bbx_0, \be}(r)=\nabla\bbu(\bbx_0+r\be)\be.
$$
If $\bbS$ stands for the unit sphere in $\R^N$, we claim that
\begin{equation}\label{polar}
\int_\bbS\int_0^\rho |u'_{\bbx_0, \be}(r)|\,dr\,d\be
\end{equation}
is finite. Indeed, by hypothesis,
$$
\int_{\bbB_\rho(\bbx_0)}|\bbx-\bbx_0|^{-N}|\nabla\bbu(\bbx)(\bbx-\bbx_0)|\,d\bbx
$$
is finite. But this integral can be elementarily written precisely in the form \eqref{polar}. 
In particular, for a.e. $\be\in\bbS$, the functions $u_{\bbx_0, \be}(r)\in W^{1,1}(0, \rho)$, and hence their limits as $r\to0^+$ exist. This is precisely the self-continuity property, if for each point $\bbx_0$ of discontinuity of $\bbu$, one such direction $\be\in\bbS$ of continuity is chosen. 
\end{proof}

A clear application of this result is Example \ref{uno} in Section \ref{ejemplos}. 

\begin{corollary}
Let $\bbu$ be a Sobolev field in $\Omega\subset\R^N$ complying with the hypothesis in the preceding proposition. The ODE system
$$
\bx'(t)=\bbu(\bx(t)),\quad t\in(0, T),\quad \bx(0)=\bbx_0,
$$
admits at least one generalized solution according to Definition \ref{generalizada} for every $\bbx_0\in\Omega$, and some $T>0$. 
\end{corollary}

\section{Final comments}
One particularly important problem is that of the coherent extension to $\R^N$ of a discontinuous vector field off a certain given $N-1$ smooth manifold $\bbS$. Or even in more general terms, when and how a discontinuous vector field $\bbf$ can be extended across its discontinuity set so that it becomes self-continuous in a coherent way with its continuous environment. This looks like a pretty complicated problem, if we do not restrict attention to particular situations. Indeed, this is one of the problems that motivates the analysis in \cite{filippov}. It is clear, after Definition \ref{basica}, that if for a certain point $\bx_0$ of discontinuity of $\bbf$ there is a smooth curve $\sigma(\epsilon; \bx_0)$ such that 
$$
\sigma(0; \bx_0)=\bx_0,\quad \lim_{\epsilon\to0}\bbf(\sigma(\epsilon; \bx_0))=\sigma'(0; \bx_0),
$$
then $\bbf$ can be extended by putting
$$
\bbf(\bx_0)=\sigma'(0; \bx_0)
$$
to become self-continuous at $\bx_0$. 

We have seen in some of the examples in the last section, that if a smooth manifold $\bbS\subset\R^N$ of dimension $N-1$ separates two open regions in $\R^N$ where a vector field is continuous, but it is not so across $\bbS$, then one can define in any way the vector field $\bbf$ on $\bbS$ as long as $\bbf(\bx)$ is tangent at $\bx\in\bbS$. To show that such an extension is self-continuous, and hence there are integral curves through all points in $\bbS$, one has to resort to the more general definition of self-continuity involving curves (second part of Definition \ref{basica}). As just pointed out above, another matter is to isolate the self-continuous extension which is ``compatible" with the surrounding, i.e. the self-continuity extension favored or induced by the field $\bbf$ in a neighborhood of $\bbS$ (see again \cite{filippov}). 

Another situation of interest would be to examine the self-continuity of Sobolev fields more thoroughly, and in particular investigate when Sobolev fields admit coherent, self-continuous extensions beyond its continuity set.

\end{document}